\documentclass[a4paper, 11pt]{amsart}

\usepackage{amsmath,amsfonts,amssymb,amscd}
\usepackage{verbatim}
\usepackage{stmaryrd}
\input{xy}
\usepackage{enumitem}
\usepackage{hyperref,mdwlist}
\usepackage[english]{babel}
\usepackage{latexsym}
\usepackage{amsopn}
\usepackage[all]{xy}
\usepackage{mathrsfs}
\usepackage{float}
\usepackage{fullpage}
\usepackage{tikz}

\usepackage{graphicx}
\usepackage{tikz-cd}
\usepackage{csquotes}
\usepackage[backend=biber,style=alphabetic,maxbibnames=99]{biblatex}
\AtBeginBibliography{\scriptsize\sloppy\setlength{\bibitemsep}{0pt}}
\AtEveryBibitem{\clearfield{doi}\clearfield{eprint}\clearfield{archivePrefix}\clearfield{primaryClass}\clearfield{url}\clearfield{urldate}\clearfield{note}}
\hypersetup{hidelinks,hypertexnames=false}
\usepackage[nameinlink,capitalise]{cleveref}

\newtheorem{theorem}{Theorem}[section]
\newtheorem{proposition}[theorem]{Proposition}
\newtheorem{corollary}[theorem]{Corollary}

\theoremstyle{definition}

\newtheorem{example}[theorem]{Example}
\theoremstyle{remark}
\newtheorem{remark}[theorem]{Remark}

\newcommand{\cO}{\mathcal O}
\newcommand{\cC}{\mathcal C}
\newcommand{\CO}{\mathcal{CO}}
\newcommand{\gap}{\operatorname{gap}}
\newcommand{\St}{\operatorname{St}}
\newcommand{\ed}{\operatorname{edges}}
\newcommand{\mi}{\operatorname{min}}
\newcommand{\ma}{\operatorname{max}}
\newcommand{\mc}{\operatorname{mc}}

\title{On facet gaps of order and chain polytopes}
\author{Ghislain Fourier}
\address{Chair of Algebra and Representation Theory, RWTH Aachen University, Pontdriesch 10--16, 52062 Aachen, Germany}
\email{fourier@art.rwth-aachen.de}

\begin{document}

\begin{abstract}

We discuss three questions from a recent paper of Bhandari, Cunningham,
Morrell, Oh and Smith. We obtain an exact local formula for the facet gap
between the order and the chain polytope of a finite poset. This gives a
classification of the case $\gap(P)=2$ in terms of star elements. For
marked chain--order polytopes, the same local weights determine the facet
differences between all admissible decompositions.
\end{abstract}

\maketitle

\section{Introduction}

All posets in this note are finite. Let $P$ be a poset. Stanley introduced the order polytope $\cO(P)$ and the chain polytope $\cC(P)$ and proved that they have the same Ehrhart polynomial \cite{Stanley1986}. Hibi and Li proved
\[
  f_{n-1}(\cO(P))\leq f_{n-1}(\cC(P))
\]
and characterized the equality case \cite{HibiLi2016}. Hibi, Li, Sahara and Shikama proved the analogous inequality for edges \cite{HibiLiSaharaShikama2017}. For further work on face numbers and related $f$-vector questions we refer to \cite{FreijHollantiLundstrom2024,
AhmadFourierJoswig2026,
Mori2025,
FreijHollantiLundstromMori2026,
FreijHollantiLundstrom2025}.

Marked order and marked chain polytopes, introduced by Ardila, Bliem and Salazar \cite{ArdilaBliemSalazar2011}, were further studied in \cite{Fourier2016}, in particular with respect to Minkowski decompositions and unimodular equivalence; see also \cite{JochemkoSanyal2014,Pegel2018}. We will use the marked chain--order polytopes introduced in \cite{FangFourier2016}, which interpolate between the marked order and marked chain polytopes. This construction was subsequently placed in a continuous family of marked poset polytopes in \cite{FangFourierLitzaPegel2020}; related combinatorial properties, in particular Minkowski decompositions and reflexivity, were studied in \cite{FangFourierPegel2020}. These constructions are different from the order--chain polytopes of Hibi, Li, Li, Mu and Tsuchiya \cite{HibiEtAl2019}. This distinction is exactly what matters for the last of these questions.

Following \cite{BhandariEtAl2025} we set
\[
  \gap(P)=f_{n-1}(\cC(P))-f_{n-1}(\cO(P)).
\]
The facet numbers are
\begin{equation}\label{eq:basic-facets}
 f_{n-1}(\cO(P))=\ma(P)+\mi(P)+\ed(P),
 \qquad
 f_{n-1}(\cC(P))=\mc(P)+|P|,
\end{equation}
where $\mc(P)$ is the number of maximal chains of $P$; see \cite{HibiLi2016,BhandariEtAl2025}.

In \cite{BhandariEtAl2025}, the authors introduce the crossing number
\[
  \operatorname{cr}(v)=(u(v)-1)(d(v)-1),
\]
where $u(v)$ and $d(v)$ count maximal chains above and below $v$. They prove bounds for $\gap(P)$ and characterize the case $\gap(P)=1$ by $X$-orchids. They conclude with three questions: a classification of the case $\gap(P)=2$ \cite[Question~15]{BhandariEtAl2025}, an upper bound which takes the value $0$ on $X$-avoiding posets and $1$ on $X$-orchids \cite[Question~16]{BhandariEtAl2025}, and facet gaps for interpolating order--chain constructions \cite[Question~17]{BhandariEtAl2025}.

Questions~15--17 of \cite{BhandariEtAl2025} motivate this note. Using the star elements introduced in \cite{FangFourier2016, HibiLi2016}, we obtain
\begin{equation}\label{eq:intro-formula}
  \gap(P)=\sum_{q\in P}(d(q)-1)(r(q)-1).
\end{equation}
This can be read from \cite[Proposition~4.5 and Corollary~4.6]{FangFourier2016}. This gives the bound asked for in
\cite[Question~16]{BhandariEtAl2025} with equality and yields a
classification of the case $\gap(P)=2$ asked for in
\cite[Question~15]{BhandariEtAl2025}. We also apply the same local weights
to the marked chain--order polytopes of \cite{FangFourier2016}, which
addresses the facet-gap problem in
\cite[Question~17]{BhandariEtAl2025} for this family. We make no claim for
the different order--chain polytopes of \cite{HibiEtAl2019}.

The formula is asymmetric: it uses maximal chains below $q$ and covers above $q$. Its dual localizes the same contribution at the opposite end. For an X-orchid, these two localizations are the endpoints of the stalk.

The paper is organized as follows. In Section~\ref{sec:gap} we prove \eqref{eq:intro-formula}. In Section~\ref{sec:small} we discuss \cite[Questions~15 and~16]{BhandariEtAl2025}. Section~\ref{sec:q17} concerns \cite[Question~17]{BhandariEtAl2025}.

\medskip
\noindent\textbf{Acknowledgment.}
The author gratefully acknowledges financial support by the Deutsche Forschungsgemeinschaft (DFG, German Research Foundation) through Symbolic Tools in Mathematics and their Application (TRR 195, project-ID 286237555).\\
ChatGPT (OpenAI) was used for language editing and for computations in examples.

\section{The facet gap and star elements}\label{sec:gap}

We add a smallest element $\widehat 0$ and a largest element $\widehat 1$ to $P$ and write $\widehat P=P\sqcup\{\widehat 0,\widehat 1\}$. For $q\in P$, let
\[
  d(q)=\#\{\text{maximal chains in $\widehat P$ from $\widehat 0$ to $q$}\}
\]
and
\[
  r(q)=\#\{x\in\widehat P:q\lessdot x\}.
\]
In particular, $d(q)=1$ for a minimal element of $P$ and $r(q)=1$ for a maximal element of $P$.

Following \cite{FangFourier2016}, an element $q$ is a \emph{star element} if
\[
  d(q)\ge 2\qquad\text{and}\qquad r(q)\ge 2.
\]
We denote the set of star elements by $\St(P)$ and set
\[
  w(q)=(d(q)-1)(r(q)-1).
\]
Thus $w(q)>0$ precisely for $q\in\St(P)$.

\begin{theorem}\label{thm:gap-formula}
For every poset $P$,
\begin{equation}\label{eq:gap-formula}
  \gap(P)=\sum_{q\in P}(d(q)-1)(r(q)-1)
  =\sum_{q\in\St(P)}w(q).
\end{equation}
\end{theorem}

\begin{proof}
The statement can be read from \cite[Proposition~4.5 and Corollary~4.6]{FangFourier2016}. Indeed, the ordinary order and chain polytopes are obtained by adjoining $\widehat 0,\widehat 1$, marking them by $0,1$, and taking the two extremal admissible decompositions. For a non-empty $P$ this marked poset is regular in the sense of \cite[Definition~4.1]{FangFourier2016}, so the facet computation there applies. We give also a direct argument, which does not use marked posets.

First,
\[
  \sum_{q\in P}r(q)=\ed(P)+\ma(P),
\]
since the first summand counts the cover relations inside $P$ and every maximal element has one additional cover, namely $\widehat 1$. Moreover,
\begin{equation}\label{eq:dr-doublecount}
  \sum_{q\in P}d(q)r(q)
  =\sum_{q\in P}d(q)-\mi(P)+\mc(P).
\end{equation}
To see this, sum $d(q)$ over all covers $q\lessdot x$ in $\widehat P$. If $x\in P$ is not minimal, then
\[
  d(x)=\sum_{q\lessdot x}d(q).
\]
The terms with $x\in P$ therefore give $\sum_{x\in P}d(x)-\mi(P)$, while the covers $q\lessdot\widehat 1$ give exactly the number $\mc(P)$ of maximal chains of $P$. Hence
\begin{align*}
 \sum_{q\in P}(d(q)-1)(r(q)-1)
 &=\sum_qd(q)r(q)-\sum_qd(q)-\sum_qr(q)+|P|\\
 &=\mc(P)+|P|-\mi(P)-\ma(P)-\ed(P)\\
 &=\gap(P)
\end{align*}
by \eqref{eq:basic-facets}.
\end{proof}

\begin{remark}\label{rem:dual}
The weight in \eqref{eq:gap-formula} is asymmetric. Applying \cref{thm:gap-formula} to the dual poset gives
\begin{equation}\label{eq:dual-gap}
  \gap(P)=\sum_{q\in P}(u(q)-1)(\ell(q)-1),
\end{equation}
where $u(q)$ is the number of maximal chains from $q$ to $\widehat 1$ and $\ell(q)$ is the number of lower covers of $q$ in $\widehat P$.

For an $X$-orchid with stalk $c_1<\cdots<c_k$, the upper endpoint $c_k$ has two upper covers and at least two maximal chains below it, hence it is a star element. By \cite[Theorem~14]{BhandariEtAl2025}, an $X$-orchid has gap one, and \cref{thm:gap-formula} then forces $c_k$ to be the unique star element and to have weight one. Thus \eqref{eq:gap-formula} localizes the contribution at $c_k$. Dually, \eqref{eq:dual-gap} localizes the same contribution at $c_1$. Thus the first formula detects $c_k$ whereas the dual formula detects $c_1$. The orchid stalk connects these two localizations.
\end{remark}

\begin{example}[The $4\times 4$ grid with a unique minimum and maximum]\label{ex:grid}
The undirected $4\times 4$ grid graph admits different poset orientations. We orient the edges so that there is a unique minimal and a unique maximal vertex; equivalently, we consider the product poset $[4]\times[4]$. We draw its Hasse diagram as a diamond in \cref{fig:grid}. The four filled vertices are the star elements, and their labels indicate the pairs $(d(q),r(q))$.

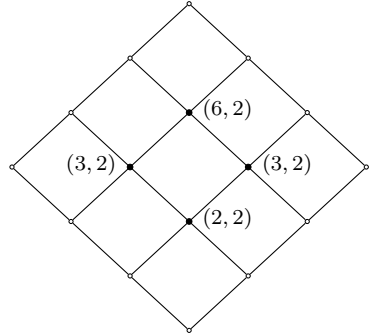
\begin{figure}[H]
\centering
\begin{tikzpicture}[x=0.78cm,y=0.72cm]
\tikzset{vtx/.style={circle,draw,fill=white,inner sep=0.55pt,line width=0.35pt},
         star/.style={circle,draw,fill=black,inner sep=0.75pt,line width=0.35pt},
         lab/.style={draw=none,fill=none,inner sep=1pt,font=\scriptsize}}
\foreach \name/\x/\y in {11/0/0,12/-1/1,21/1/1,13/-2/2,22/0/2,31/2/2,14/-3/3,23/-1/3,32/1/3,41/3/3,24/-2/4,33/0/4,42/2/4,34/-1/5,43/1/5,44/0/6}
  \node[vtx] (\name) at (\x,\y) {};
\foreach \a/\b in {11/12,11/21,12/13,12/22,21/22,21/31,13/14,13/23,22/23,22/32,31/32,31/41,14/24,23/24,23/33,32/33,32/42,41/42,24/34,33/34,33/43,42/43,34/44,43/44}
  \draw (\a)--(\b);
\foreach \name in {22,23,32,33}
  \node[star] at (\name) {};
\node[lab,anchor=west] at (0.17,2.06) {$(2,2)$};
\node[lab,anchor=east] at (-1.18,3.04) {$(3,2)$};
\node[lab,anchor=west] at (1.18,3.04) {$(3,2)$};
\node[lab,anchor=west] at (0.17,4.06) {$(6,2)$};
\end{tikzpicture}
\caption{The $4\times4$ grid oriented with a unique minimum and maximum, drawn as a diamond. The filled vertices are the star elements; the labels show their types $(d(q),r(q))$.}
\label{fig:grid}
\end{figure}

We have $|P|=16$, $\ed(P)=24$, one minimal and one maximal element. Thus $f_{15}(\cO(P))=26$. There are $\binom{6}{3}=20$ maximal chains, so $f_{15}(\cC(P))=36$ and $\gap(P)=10$. The four star weights are
\[
  1,\qquad 2,\qquad 2,\qquad 5,
\]
which add up to $10$.
\end{example}

\section{Gaps one and two}\label{sec:small}

The exact formula gives a classification of the first possible gaps in terms of star elements. In particular, part~(iii) below answers \cite[Question~15]{BhandariEtAl2025} at the level of the local star data.

\begin{corollary}\label{cor:smallgap}
\begin{enumerate}[label=\textup{(\roman*)}]
  \item $\gap(P)=0$ if and only if $\St(P)=\varnothing$.
  \item $\gap(P)=1$ if and only if there is exactly one star element $q$ and
  \[
     (d(q),r(q))=(2,2).
  \]
  \item $\gap(P)=2$ if and only if exactly one of the following occurs:
  \begin{enumerate}[label=\textup{(\alph*)}]
    \item there are exactly two star elements $q_1,q_2$, and both are simple:
    \[
       (d(q_i),r(q_i))=(2,2),\qquad i=1,2;
    \]
    \item there is exactly one star element $q$, and
    \[
       (d(q),r(q))\in\{(2,3),(3,2)\}.
    \]
  \end{enumerate}
\end{enumerate}
\end{corollary}

\begin{proof}
By \cref{thm:gap-formula},
\[
  \gap(P)=\sum_{q\in\St(P)} w(q),
  \qquad
  w(q)=(d(q)-1)(r(q)-1)\in\mathbb Z_{>0}.
\]
This immediately gives (i).

For (ii), the sum is equal to $1$ if and only if there is exactly one
star element $q$ of weight one. Since
\[
  (d(q)-1)(r(q)-1)=1,
\]
this is equivalent to $(d(q),r(q))=(2,2)$.

For (iii), either there is one star element of weight two, or there are
exactly two star elements, both of weight one. In the first case
\[
  (d(q)-1)(r(q)-1)=2,
\]
hence
\[
  (d(q),r(q))\in\{(2,3),(3,2)\}.
\]
In the second case
\[
  (d(q_i)-1)(r(q_i)-1)=1,\qquad i=1,2,
\]
and therefore
\[
  (d(q_i),r(q_i))=(2,2),\qquad i=1,2.
\]
The converse is immediate from \cref{thm:gap-formula}.
\end{proof}

In case \textup{(iii)(a)}, the two simple stars may be incomparable or comparable. In the comparable case they can be intertwined along a linear chain. In case \textup{(iii)(b)}, one star contributes weight two; here $d(q)=3$ counts three maximal chains below $q$ and does not require three lower covers. Figure~\ref{fig:gap2} shows schematic realizations.

\begin{figure}[H]
\centering
\begin{minipage}[t]{0.32\textwidth}
\centering
\begin{tikzpicture}[x=0.44cm,y=0.46cm]
\tikzset{vtx/.style={circle,draw,fill=white,inner sep=0.45pt,line width=0.32pt},
         star/.style={circle,draw,fill=black,inner sep=0.65pt,line width=0.32pt},
         lab/.style={draw=none,fill=none,inner sep=0.8pt,font=\scriptsize}}

\node[vtx] (a1) at (-2.5,0.9) {};
\node[vtx] (a2) at (-1.3,0.9) {};
\node[vtx] (c1) at (-1.9,1.8) {};
\node[vtx] (m1) at (-1.9,2.6) {};
\node[star] (q1) at (-1.9,3.4) {};
\node[vtx] (u11) at (-2.5,4.3) {};
\node[vtx] (u12) at (-1.3,4.3) {};

\node[vtx] (b1) at (1.3,0.9) {};
\node[vtx] (b2) at (2.5,0.9) {};
\node[vtx] (c2) at (1.9,1.8) {};
\node[vtx] (m2) at (1.9,2.6) {};
\node[star] (q2) at (1.9,3.4) {};
\node[vtx] (u21) at (1.3,4.3) {};
\node[vtx] (u22) at (2.5,4.3) {};

\draw (a1)--(c1)--(m1)--(q1)--(u11);
\draw (a2)--(c1);
\draw (q1)--(u12);
\draw (b1)--(c2)--(m2)--(q2)--(u21);
\draw (b2)--(c2);
\draw (q2)--(u22);
\node[lab,anchor=south] at (-1.6,2.8) {$q_1\;\ (2,2)$};
\node[lab,anchor=south] at (2.3,2.8) {$q_2\;\; (2,2)$};
\end{tikzpicture}

\smallskip
Incomparable simple stars
\end{minipage}\hfill
\begin{minipage}[t]{0.32\textwidth}
\centering
\begin{tikzpicture}[x=0.47cm,y=0.46cm]
\tikzset{vtx/.style={circle,draw,fill=white,inner sep=0.45pt,line width=0.32pt},
         star/.style={circle,draw,fill=black,inner sep=0.65pt,line width=0.32pt},
         lab/.style={draw=none,fill=none,inner sep=0.8pt,font=\scriptsize}}
\node[vtx] (a) at (-1.0,0) {};
\node[vtx] (b) at (1.0,0) {};
\node[vtx] (c1) at (0,0.9) {};
\node[vtx] (c2) at (0,1.7) {};
\node[star] (q1) at (0,2.5) {};
\node[vtx] (s) at (1.55,3.3) {};
\node[vtx] (m1) at (0,3.3) {};
\node[vtx] (m2) at (0,4.1) {};
\node[star] (q2) at (0,4.9) {};
\node[vtx] (u) at (-0.9,5.85) {};
\node[vtx] (v) at (0.9,5.85) {};
\draw (a)--(c1)--(c2)--(q1)--(m1)--(m2)--(q2)--(u);
\draw (b)--(c1);
\draw (q1)--(s);
\draw (q2)--(v);
\node[lab,anchor=east] at (-0.12,2.5) {$q_1$};
\node[lab,anchor=west] at (0.12,2.3) {$(2,2)$};
\node[lab,anchor=east] at (-0.12,4.9) {$q_2$};
\node[lab,anchor=west] at (0.12,4.7) {$(2,2)$};
\end{tikzpicture}

\smallskip
Comparable, intertwined stars
\end{minipage}\hfill
\begin{minipage}[t]{0.32\textwidth}
\centering
\begin{tikzpicture}[x=0.47cm,y=0.46cm]
\tikzset{vtx/.style={circle,draw,fill=white,inner sep=0.45pt,line width=0.32pt},
         star/.style={circle,draw,fill=black,inner sep=0.65pt,line width=0.32pt},
         lab/.style={draw=none,fill=none,inner sep=0.8pt,font=\scriptsize}}
\node[vtx] (a) at (-1.35,0) {};
\node[vtx] (b) at (0,0) {};
\node[vtx] (c) at (1.35,0) {};
\node[vtx] (c1) at (0,0.95) {};
\node[vtx] (c2) at (0,1.75) {};
\node[vtx] (c3) at (0,2.55) {};
\node[star] (q) at (0,3.35) {};
\node[vtx] (u) at (-0.9,4.45) {};
\node[vtx] (v) at (0.9,4.45) {};
\draw (a)--(c1)--(c2)--(c3)--(q)--(u);
\draw (b)--(c1);
\draw (c)--(c1);
\draw (q)--(v);
\node[lab,anchor=east] at (-0.12,3.2) {$q$};
\node[lab,anchor=west] at (0.12,3.2) {$(3,2)$};
\end{tikzpicture}

\smallskip
One star of weight two
\end{minipage}
\caption{Schematic realizations of Corollary~\ref{cor:smallgap}(iii). Left: two incomparable simple stars. Middle: two comparable simple stars; here the crossing regions are intertwined along a linear chain. Right: one star of type $(3,2)$; the dual type is $(2,3)$. Linear chains inside the crossing regions do not change the star weights.}
\label{fig:gap2}
\end{figure}
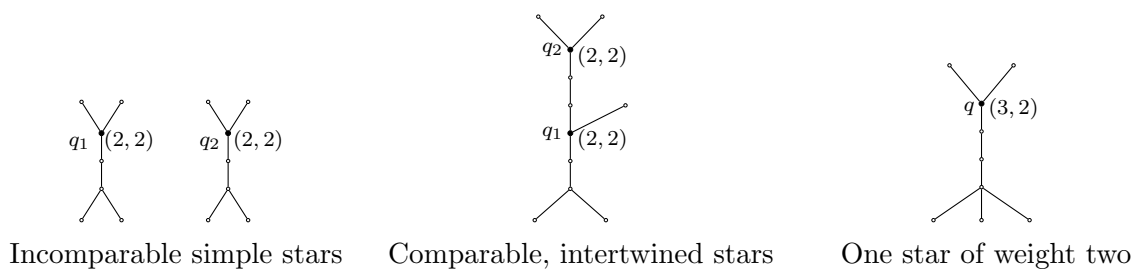

Theorem~\ref{thm:gap-formula} also answers \cite[Question~16]{BhandariEtAl2025}: the right-hand side of \eqref{eq:gap-formula} is the requested upper bound, in fact with equality. By \cite[Theorems~12 and~14]{BhandariEtAl2025}, it takes the values 0 on X-avoiding posets and 1 on X-orchids.
The star description also recovers the endpoint picture of an orchid: by \cref{cor:smallgap}(ii), an $X$-orchid has one star of type $(2,2)$; as observed in Remark~\ref{rem:dual}, this star is the upper endpoint $c_k$ of the stalk, while the dual formula selects $c_1$.

\section{Chain--order polytopes and Question~17}\label{sec:q17}

In \cite[Question~17]{BhandariEtAl2025}, the authors ask about facet
gaps for polytopes interpolating between order and chain polytopes.
We consider here the marked chain--order polytopes of
\cite{FangFourier2016}.

Let $(P,A,\lambda)$ be a regular marked poset and
\[
  P\setminus A=U_1\sqcup U_2
\]
an admissible decomposition, i.e.\ $U_2$ is an order ideal.
Here $U_1$ is the order part and $U_2$ the chain part; see
\cite[Definition~1.5]{FangFourier2016}.

We use the notation of \cite{FangFourier2016}: $p\to q$ means that
$q$ covers $p$, and $\rightsquigarrow q$ denotes the set of maximal
chains ending in $q$. Thus
\[
  |q\to|=\#\{p:q\to p\},
  \qquad
  |\rightsquigarrow q|
  =\#\{\text{maximal chains ending in }q\},
\]
and for a star element $q$ we set
\[
  w(q)=(|q\to|-1)(|\rightsquigarrow q|-1).
\]
Let $N=|P\setminus A|$.

\begin{proposition}\label{prop:q17-ff}
For every admissible decomposition $(U_1,U_2)$ of a regular marked poset,
\begin{equation}\label{eq:co-facet-formula}
 f_{N-1}\bigl(\CO_{U_1,U_2}(\lambda)\bigr)
 =f_{N-1}\bigl(\cO_{P,A}(\lambda)\bigr)
 +\sum_{q\in U_2\cap\St(P)}w(q).
\end{equation}
Consequently, facet differences between two admissible decompositions are obtained by subtracting the corresponding sums of star weights.
\end{proposition}

\begin{proof}
Starting from the marked order polytope $(U_1,U_2)=(P\setminus A,\varnothing)$, move the elements of the order ideal $U_2$ to the chain part one at a time, in an order compatible with the poset. If the moved element is not a star, the set of star elements in the chain part does not change, and \cite[Theorem~4.7]{FangFourier2016} shows that the two consecutive polytopes are unimodular equivalent. If a star $q$ is moved, \cite[Corollary~4.6]{FangFourier2016} gives the exact increase in the number of facets as
\[
  (|q\to|-1)(|\rightsquigarrow q|-1)=w(q).
\]
Summing these increments gives \eqref{eq:co-facet-formula}.
\end{proof}

For the ordinary order/chain situation, obtained by adjoining $\widehat 0,\widehat 1$ and marking them by $0,1$, one has $|q\to|=r(q)$ and $|\rightsquigarrow q|=d(q)$. This is in fact a regular marked poset. Hence the extreme choice $U_2=P$ recovers \cref{thm:gap-formula}; intermediate admissible decompositions simply record which star weights have entered the chain part. In this sense, the facet-gap part of \cite[Question~17]{BhandariEtAl2025} already has an exact answer for the family of \cite{FangFourier2016}.

The description is particularly simple for the small-gap configurations of Section~\ref{sec:small}. If $P$ is an $X$-orchid with stalk $c_1<\cdots<c_k$, its unique star is $c_k$ and has weight one. Since $U_2$ is an order ideal,
\[
  c_k\in U_2
  \quad\Longleftrightarrow\quad
  \{c_1,\ldots,c_k\}\subseteq U_2.
\]
Thus the facet number increases by one exactly when the whole stalk has entered the chain part. The star formula records this change at the endpoint $c_k$, whereas the orchid records the complete linear crossing region.

For gap two, Figure~\ref{fig:gap2} also describes how the increments can occur. If the two simple stars $q_1,q_2$ are incomparable, their weight-one contributions can enter independently. If $q_1<q_2$, admissibility gives
\[
  q_2\in U_2\;\Longrightarrow\;q_1\in U_2,
\]
so the contribution of $q_1$ must occur first; the middle picture of Figure~\ref{fig:gap2} shows one intertwined realization. If there is one star of weight two, its passage to the chain part produces a single jump of size two. The cases $(d,r)=(3,2)$ and $(2,3)$ are dual.

The order--chain polytopes of \cite{HibiEtAl2019} are different. There the cover relations are decomposed as
\[
 E(P)=E_o\sqcup E_c,
\]
and the construction combines the order polytope associated with $E_o$ and the chain polytope associated with $E_c$. This is an edge decomposition, not the vertex decomposition above; see \cite[Remark~1.9]{FangFourier2016}. Therefore \cref{prop:q17-ff} applies to the marked chain--order polytopes of \cite{FangFourier2016}, but not to the Hibi--Li--Li--Mu--Tsuchiya order--chain polytopes. We do not make a claim about crossing-number bounds for the latter family.

\printbibliography

\end{document}